\documentclass[11pt,a4paper,oneside,reqno]{amsart}

\usepackage[all]{xy}
\SelectTips{cm}{}  \calclayout
\usepackage{amsfonts,amssymb,amscd,amsmath,enumerate,verbatim,calc}

\usepackage{amscd,amssymb,amsopn,amsmath,amsthm,graphics,amsfonts,enumerate,verbatim,calc}
\usepackage[dvips]{graphicx}
\usepackage[colorlinks=true,linkcolor=blue,citecolor=blue]{hyperref}
\input xy
\xyoption{all}

\newcommand{\rt}{\rightarrow}
\newcommand{\lrt}{\longrightarrow}

\newcommand{\pa}{\partial}

\newcommand{\C}{\mathbf{C} }
\newcommand{\D}{\mathbf{D} }
\newcommand{\K}{\mathbf{K} }
\newcommand{\X}{\mathbf{X} }
\newcommand{\Y}{\mathbf{Y} }
\newcommand{\G}{\mathbf{G} }
\newcommand{\I}{\mathbf{I} }

\newcommand{\F}{\mathbf{F} }

\newcommand{\CO}{\mathcal{O}}

\newcommand{\T}{\mathbf{T} }

\newcommand{\Z}{\mathbb{Z} }

\newcommand{\CE}{\mathcal{E}}

\newcommand{\CC}{\mathcal{C} }

\newcommand{\CF}{\mathcal{F} }
\newcommand{\CG}{\mathcal{G} }

\newcommand{\CL}{\mathcal{L} }
\newcommand{\CH}{\mathcal{H} }

\newcommand{\CK}{\mathcal{K} }
\newcommand{\CP}{\mathcal{P} }
\newcommand{\CQ}{\mathcal{Q} }

\newcommand{\CS}{\mathcal{S} }
\newcommand{\CT}{\mathcal{T} }

\newcommand{\CJ}{\mathcal{J} }

\newcommand{\Inj}{{\rm{Inj}}}
\newcommand{\Prj}{{\rm{Proj}}}
\newcommand{\Flat}{{\rm{Flat}}}

\newcommand{\KCOFR}{{\K({\rm{Cof}} R)}}
\newcommand{\KFX}{{\mathbf{K}({\Flat} X)}}

\newcommand{\KCOFX}{{\K({\rm{Cof}} X)}}

\newcommand{\KPFX}{{\mathbf{K}_{\rm{pac}}({\Flat} X)}}
\newcommand{\DPFX}{{\mathbf{D}_{\rm{pac}}({\Flat} X)}}

\newcommand{\DPABX}{{\mathbf{D}_{\rm{pur}}({\mathrm{Abs}} X)}}

\newcommand{\DPFR}{{\mathbf{D}_{\rm{pur}}({\Flat} R)}}

\newcommand{\KPACX}{{\mathbf{K}_{\rm{pac}}(X)}}

\newcommand{\DPACX}{{\mathbf{D}_{\rm{pur}}(X)}}

\newcommand{\KPAIX}{{\mathbf{K}_{\rm{pac}}({\Inj} X)}}

\newcommand{\KABX}{{\mathbf{K}({\mathrm{Abs}} X)}}
\newcommand{\KPABX}{{\mathbf{K}_{\rm{pac}}({\mathrm{Abs}} X)}}

\newcommand{\KPFR}{{\mathbf{K}_{\rm{pac}}({\Flat} R)}}

\newcommand{\CPX}{{\mathbf{C}_{\rm{pac}}(X)}}

\newcommand{\CPFR}{{\mathbf{C}_{\rm{pac}}({\Flat} R)}}
\newcommand{\CPFX}{{\mathbf{C}_{\rm{pac}}({\Flat} X)}}

\newcommand{\CPABX}{{\mathbf{C}_{\rm{pac}}({\mathrm{Abs}} X)}}
\newcommand{\KPR}{{\mathbf{K}({\Prj}  R)}}

\newcommand{\KFR}{{\mathbf{K}({\Flat}  R)}}
\newcommand{\KIX}{{\mathbf{K}({\Inj}  X)}}
\newcommand{\KPIX}{{\mathbf{K}({\mathrm{Pinj}}  X)}}

\newcommand{\Hom}{{\rm{Hom}}}
\newcommand{\Ext}{{\rm{Ext}}}

\newcommand{\limt}{\underset{\underset{j \in J}{\longrightarrow}}{\lim}}

\newtheorem{theorem}{Theorem}[section]
\newtheorem{corollary}[theorem]{Corollary}
\newtheorem{lemma}[theorem]{Lemma}
\newtheorem{proposition}[theorem]{Proposition}

\theoremstyle{definition}

\newtheorem{example}[theorem]{Example}

\newtheorem{remark}[theorem]{Remark}
\newtheorem{s}[theorem]{}

\theoremstyle{plain}
\newtheorem{stheorem}{Theorem}[subsection]

\newtheorem{scorollary}[stheorem]{Corollary}

\newtheorem{sproposition}[stheorem]{Proposition}

\theoremstyle{definition}
\newtheorem{sdefinition}[stheorem]{Definition}

\newtheorem{sremark}[stheorem]{Remark}

\numberwithin{equation}{section}
\begin{document}

\title[The Pure derived category of quasi-coherent sheaves]{The Pure derived category of quasi-coherent sheaves}
\author[ Esmaeil Hosseini]{Esmaeil Hosseini  }

\address{Department of Mathematics, Shahid Chamran University of
Ahvaz, P.O.Box: 61357-83151, Ahvaz, Iran.}
\email{e.hosseini@scu.ac.ir}

\begin{abstract}
Let $X$ be a quasi-compact and quasi-separated (not necessarily
semi-separated) scheme. The category $\mathfrak{Qco}X$ of all
quasi-coherent sheaves of $\CO_X$-modules has several different pure
derived categories. Recently, categorical pure derived categories of
$X$ have been studied in more details. In this work, we focus on the
geometrical purity and find a replacement for the geometrical pure
derived category of $X$.

\end{abstract}
\maketitle
\section{Introduction}

Assume that $\CC$ is an abelian category. The derived category of
$\CC$ was invented by Grothendieck and his student Verdier to
simplify the theory of derived functors on $\CC$ (\cite{Ver}). It
was obtained from the category of complexes in $\CC$ by formally
inverting all quasi-isomorphisms. This shows that the definition  of
 derived category is closely related to the concept of localization
of triangulated categories, i.e. it is defined as the localization
of the homotopy category of $\CC$ with respect to the class of all
quasi-isomorphisms. But, this method is not applicable for
non-abelian categories. In fact, if $\mathbb{E}$ is an exact
category (not necessarily abelian),  the unbounded derived category
$\D(\mathbb{E})$ of $\mathbb{E}$ can not be defined in the same way.
In 1990, Neeman showed that $\D(\mathbb{E})$ can be determined by
the quotient of the homtopy category of $\mathbb{E}$ modulo its
thick (or $\acute{\mathrm{e}}$paisse) subcategory consisting of all
exact complexes (\cite{Ne90}). This approach is independent of the
choice of exact structure and is suitable for a category with
various exact structures. Especially, it is suitable for the
category $\mathfrak{Qco}X$ of all quasi-coherent sheaves over a
given scheme $X$. In non-affine case,  the category
$\mathfrak{Qco}X$ admits three different exact structures. One of
them is defined by the abelian exact sequences and the others are
defined by pure exact sequences (see Example \ref{mor1}).

Assume that $X$ is a quasi-compact and quasi-separated scheme. By
\cite{TT}, $\mathfrak{Qco}X$ is a locally finitely presented
Grothendieck category and so,  it is a natural framework to define
the categorical pure exact sequences (see \cite{Cr94}). Recall that,
a short exact sequence $\CL$ of quasi-coherent $\CO_X$-modules is
called categorical pure if for any finitely presented quasi-coherent
$\CO_X$-module $\CG$, $\mathrm{Hom}_{\CO_X}(\CG,\CL)$ is exact. The
pure derived category with respect to this exact structure, denoted
by $\D_{\mathrm{cpur}}(X)$, is defined by \cite{Ne90}. In affine
case, when $R$ is a ring, the pure derived category of $R$ was first
introduced in \cite{CH02} by formally inverting pure homological
isomorphisms. In non-affine case, $\D_{\mathrm{cpur}}(X)$ is studied
in more details by several authors (see \cite{K12}, \cite{Sto14},
\cite{Gi15}, \cite{ZH16}). Another pure exact structure on
$\mathfrak{Qco}X$ is defined by tensor pure exact sequences. Recall
that, a short exact sequence $\CL$ of quasi-coherent $\CO_X$-modules
is called tensor pure if for any quasi-coherent $\CO_X$-module
$\CG$, $\CG\otimes_{\CO_X}\CL$ is exact. The pure derived category
with respect to this exact structure is defined by \cite{Ne90} and
denoted by $\DPACX$. This category was first appeared in
\cite{EGO14}. These pure derived categories encouraged us to ask the
following question.

\vspace{.5cm} \textbf{Question:} When $\DPACX$ and
$\D_{\mathrm{cpur}}(X)$ are coincide? \vspace{.5cm}

This  is our motivation on the present  work. To find an answer to
this question, we should compare the pure exact structures on
$\mathfrak{Qco}X$. In Example \ref{mor1}, if $X$ is a projective
space over a commutative noetherian ring $R$, we find a finitely
presented quasi-coherent $\CO_X$-module which is not tensor pure
projective. Consequently,  \textbf{the} \textbf{tensor}
\textbf{pure} \textbf{exact} \textbf{structure} \textbf{does}
\textbf{not} \textbf{coincide} \textbf{with} \textbf{the}
\textbf{categorical} \textbf{pure} \textbf{exact} \textbf{structure}
and so,  $\DPACX$ and $\D_{\mathrm{cpur}}(X)$ are different. This
encouraged us to concentrate our study on \textbf{tensor}
\textbf{pure} \textbf{derived} \textbf{categories}. In Section 2, we
prove that any pure acyclic complex of pure injective
$\CO_X$-modules is contractible. In Section 3, we prove that the
homotopy category of all pure injective quasi-coherent
$\CO_X$-modules  and $\DPACX$ are equivalent.

Before starting, let us fix some notations and definitions. In this
paper, $X$ is a quasi-compact and quasi-separated scheme,
$\CO_X$-modules are quasi-coherent sheaves of $\CO_X$-modules and
$\mathfrak{Qco}X$ is the category of all $\CO_X$-modules.

\begin{s}{\sc Exact categories.} Let $\mathbb{E}$ be an additive category. A sequence
$\xymatrix@C0.7pc@R0.9pc{ \X\ar[r]^f&\Y\ar[r]^g&\mathbf{Z}}$ in
$\mathbb{E}$ is said to be a \textit{conflation} if $f$ is the
kernel of $g$ and $g$ is the cokernel of $f$. A map such as $f$ is
called an \textit{inflation} and $g$ is called a \textit{deflation}.
Let $\CE$ be a class of conflations in $\mathbb{E}$. The pair
$(\mathbb{E}, \CE)$ is an exact category if the following axioms
hold.
\begin{itemize}
\item [(i)]  $\CE$ is closed under isomorphisms.
\item [(ii)] $\CE$ contains  all   split exact sequences.
\item [(iii)] Deflations (resp. Inflations) are closed under composition.
\item [(iv)] All pullbacks (resp. pushouts) of  deflations (resp. inflations)  exist.
\item [(v)] Deflations (resp. Inflations) are stable under pullbacks (resp. pushouts).
\end{itemize}

For more details, see \cite[Appendix A]{Ke90} or \cite[Section
3]{Sto13}.

An exact category $\mathbb{E}$ is called $\textit{efficient}$ if

\begin{itemize}
\item [(i)] Every section  in $\mathbb{E}$
has a cokernel or, equivalently, every retraction in $\mathbb{E}$
has a kernel (\cite[Lemma 7.1]{Bu10}).
\item [(ii)] Arbitrary transfinite compositions of inflations exist
and are themselves inflations.
\item [(iii)] Every object of $\mathbb{E}$ is small relative to the
class of all inflations.
\item [(iv)] $\mathbb{E}$ admits a generator.

\end{itemize}

\end{s}

\begin{s}{\sc Orthogonality in exact categories.} Let $\mathbb{E}$  be an exact category  and $\X, \Y\in \mathbb{E}$.
The set of all equivalence classes of conflations  $ \Y\rightarrow
\mathbf{P} \rightarrow \X $ in $\mathbb{E}$ is denoted by
$\Ext^1_{\mathbb{E}}(\X,\Y)$. A class $\mathbb{Y}$ is said to be the
\textit{right} \textit{orthogonal} of a class $\mathbb{X}$ if
$$\mathbb{Y}=\mathbb{X}^\perp:=\{\mathbf{B} \in \mathbb{E} \ | \ \Ext_{\mathbb{E}}^1(\X,\mathbf{B})=0, \ {\rm{for \ all}}
\ \X \in \mathbb{X} \}.$$

The \textit{left} \textit{orthogonal} is defined dually. A pair
$(\mathbb{X}, \mathbb{Y})$ of classes in $\mathbb{E}$ is called a
\textit{complete} \textit{cotorsion} \textit{theory} if

\begin{itemize}
\item [(i)] $\mathbb{X}^\perp=\mathbb{Y}$ and
$^\perp\mathbb{Y}=\mathbb{X}$.
\item [(ii)] Any object in $\mathbb{E}$
has a special $\mathbb{Y}$-preenvelope and a special
$\mathbb{X}$-precover.

\end{itemize}

Recall that, an object $\mathbf{E} \in \mathbb{E}$ has a
\textit{special} $\mathbb{X}$-\textit{precover} (resp.
\textit{special} $\mathbb{Y}$-\textit{preenvelope}) if there exists
a conflation $\Y' \rt \X' \rt \mathbf{E} $ (resp. $ \mathbf{E} \rt
\Y' \rt \X' $), where $\X'\in \mathbb{X}$ and $\Y' \in \mathbb{Y}$.
 Also, $(\mathbb{X},
\mathbb{Y})$ is \textit{cogenerated} by a set, if there is a set
$\CT\subseteq\mathbb{X}$ such that $\CT^\perp= \mathbb{X}^\perp$.
\end{s}

\textbf{Setup:} Throughout  this paper, \textbf{all} \textbf{pure}
\textbf{exact} \textbf{sequences} \textbf{are} \textbf{tensor}
\textbf{pure}, all \textbf{pure} \textbf{injective} (resp.
\textbf{pure} \textbf{projective}, \textbf{flat}) objects are
\textbf{tensor} pure injective (resp. pure projective, flat) and all
\textbf{equivalences} between triangulated categories are
\textbf{triangle} \textbf{equivalences}.

\section{Pure acyclic  complexes of $\CO_X$-modules}

Recall that,  $X$ is called $\textit{quasi-separated}$ if the
intersection of any two quasi-compact open subsets is quasi-compact.
Let $\mathfrak{U}=\{U_i=\mathrm{Spec}(R_i)\}_{i=0}^n$ be a
quasi-separating open  cover of $X$. First, we compare the
categorical and the geometrical notions of purity in
$\mathfrak{Qco}X$. Note that, any of these notions induces a class
of pure injectives (resp. pure projectives). Categorical pure
injectives are discussed in \cite{Her03} and  pure injective objects
are discussed in \cite{EEO16}. In \cite[Theorem 4.10]{EEO16}, the
author proved that $\mathfrak{Qco}X$ has enough pure injective
objects. But, the pure projective objects have a different fate. By
\cite{TT}, there is a set $S$ of finitely presented objects in
$\mathfrak{Qco}X$ such that any $\CO_X$-module $\CG$  is a direct
limit of elements in $S$. Then, there exists a categorical pure
exact sequence $\xymatrix@C-0.7pc{0\ar[r]&\CK\ar[r]&\oplus_{\CF\in
S}\CF\ar[r]&\CG\ar[r]&0 }$ where  $\oplus_{\CF\in S}\CF$ is a
categorical pure projective $\CO_X$-module. Therefore,
$\mathfrak{Qco}X$ has enough categorical pure projective objects.
But, in general case, we don't know whether $\mathfrak{Qco}X$ has
enough pure projective objects? However, whether this is the case or
not, we give some examples of  non-pure projective  $\CO_X$-modules.
These examples show that the pure exact structures on
$\mathfrak{Qco}X$ are different.

\begin{example}\label{mor1}
Assume that $X = \mathrm{Proj} R[x_0,x_1,...,x_n]$ is the projective
$n$-space over a noetherian commutative ring $R$. The structure
sheaf $\CO_X$ is not pure projective. If it is pure projective then
any short exact sequence ending in $\CO_X$ splits ($\CO_X$ is flat
and so any short exact sequences ending in $\CO_X$ is pure). It
follows that $\CO_X$ is projective and so by \cite[Proposition III.
6.3]{H}, $\Ext_{\CO_X}^i(\CO_X,\CG)\cong\mathrm{H}^i(X,\CG)=0$ for
any $\CO_X$ module $\CG$ and any $i>0$. But, this is a contradiction
with \cite[Theorem III. 3.7]{H} and \cite[Theorem III. 5.1]{H}.
Consequently, $\CO_X$ is a non-pure projective finitely presented
$\CO_X$-module.
\end{example}
\begin{example}\label{mor01}
Vector bundles over $X$ are not pure projective. Since any pure
projective flat $\CO_X$-module is projective. We known that
projective objects are rare in $\mathfrak{Qco}X$ (\cite[Ex. III.
6.2]{H}) and any vector bundle is  flat $\CO_X$-module.
\end{example}

Let $\CE$ be the class of all pure exact sequences in
$\mathfrak{Qco}X$. This class induces an exact structure on the
category $\C(X)$ of all complexes (complexes are write
cohomologically) in $\mathfrak{Qco}X$ as follows. A sequence
$$\xymatrix@C-0.7pc{ 0\ar[r]&\T\ar[r]^f&\F\ar[r]^g&\G\ar[r]&0}$$ of
complexes of $\CO_X$-modules is called a conflation in $\C(X)$ if it
is degree-wise pure exact, i.e. for any $n\in\Z$,
$$\xymatrix@C-0.7pc{
0\ar[r]&\T^n\ar[r]^{f^n}&\F^n\ar[r]^{g^n}&\G^n\ar[r]&0}\in\CE.$$
This  exact structure on $\C(X)$ is called pure exact structure. In
the present work, we select this \textbf{exact} structure and prove
that any pure acyclic complex of pure injectives is contractible.
Recall that a complex $\X$ of $\CO_X$-modules is called acyclic if
all cohomological groups are trivial and it is called pure acyclic
if for any $\CO_X$-module $\CG$, $\X\otimes_{\CO_X}\CG$ is acyclic.
In \cite{Sto14}, the author proved the following result.
\begin{proposition}\label{pesaram}
Any categorically pure acyclic complex of $R$-modules is the direct
limit of contractible complexes.
\end{proposition}
\begin{proof}
See \cite[Lemma 4.14]{Sto14} or \cite[Proposition 2.2]{Em16}.
\end{proof}

It can not be the case in general that every tensor pure acyclic
complex is the direct limit of contractible complexes. The reason is
that a direct limit of contractible complexes is always
categorically pure acyclic, and there are in general fewer
categorically pure  acyclic complexes than tensor pure ones. But, we
try to prove that any pure acyclic complex of pure injective
$\CO_X$-modules is contractible.

\begin{lemma}\label{pur112}
Any direct limit of pure exact sequences is pure.
\end{lemma}
\begin{proof}
Let $\{\CL_i,f_{ij}\}_{i,j\in J}$ be a direct system of pure exact
sequences of $\CO_X$-modules. The isomorphism
$(\limt\CL_j)\otimes_{\CO_X}\CG\cong\limt(\CL_j\otimes_{\CO_X}\CG)$
implies the purity of $\limt\CL_j$.
\end{proof}
\begin{lemma}\label{pur111}
Let $U$ be an affine open subset of $X$ and $f:U\lrt X$ be the
inclusion. Then $f_*:\mathfrak{Qco}U\lrt\mathfrak{Qco}X$ preserves
pure exact sequences.
\end{lemma}
\begin{proof}
Let $\CL$ be a pure exact sequence of $\CO_U$-modules. By
Proposition \ref{pesaram}, $\CL=\limt\CL_j$ where for any $j$,
$\CL_j$ is split exact. We know that  $f_*$ preserves direct limits
(\cite[pp. 410, Lemma B.6]{TT}), i.e., we have the following
isomorphism
\begin{align}\label{limitd}
  f_*(\CL)=f_*(\limt\CL_j)=\limt(f_*(\CL_j))
\end{align}
where $f_*(\CL_j)$ splits for each $j$. Then, by Lemma \ref{pur112},
$f_*(\CL)$ is pure.
\end{proof}

Suppose that  $\CF$ is an $\CO_X$-module.  The $p$-th
$\check{\mathrm{C}}$ech $\CO_X$-module of $\CF$  with respect to the
cover $\mathfrak{U}$ of $X$ is defined by
$\mathfrak{C}^{p}(\mathfrak{U},\CF)=\oplus_{i_0<...<i_p}f_*(\CF|_{U_{i_0,...,i_p}})$
over sequences $i_0 < ... < i_p$ of length $p$ ($0\leq p\leq n$)
whenever $U_{i_0, ... ,i_p} = U_{i_0}\cap ... \cap U_{i_p} $ and
$f:U_{i_0, ... ,i_p}\lrt X$ is the inclusion.

\begin{lemma}\label{Cech0011}
Let $\F$ be a pure acyclic complex of $\CO_X$-modules. Then, for any
$0\leq p\leq n$, $\mathfrak{C}^p(\mathfrak{U},\F)$ is a  direct
limit of contractible complexes.
\end{lemma}
\begin{proof}
Without loss of generality, we prove the assertion for $p=1$. Since
$U_0$ and $U_1$ are affine open subset of $X$ then, $\F(U_0)$ is a
pure acyclic complex of $R_0$-modules and $\F(U_1)$ is a pure
acyclic complex of $R_1$-modules. By Proposition \ref{pesaram}, they
are direct limits of contractible complexes. So,
$\F|_{U_0}=\widetilde{\F(U_0)}$ and $\F|_{U_1}=\widetilde{\F(U_1)}$
are direct limits of contractible complexes. Therefore,
$\F|_{U_0\cap U_1}=(\F|_{U_0})|_{U_1}=(\F|_{U_1})|_{U_0}$ are direct
limits of contractible complexes ($U_0\cap U_1$ is not necessarily
affine). By \cite[pp. 410, Lemma B.6]{TT}, $f_*$ preserves direct
limits. It follows that $f_*\F|_{U_0}$, $f_*\F|_{U_0}$ and
$f_*\F|_{U_0\cap U_1}$ are direct limit of contractible complexes.
By a similar argument used for the case $p=1$, we deduce that for
any $0\leq p\leq n$, $\F|_{U_{i_0}}$, $\F|_{U_{i_1}}$, ...,
$\F|_{U_{i_p}}$, ..., $\F|_{U_{i_0, ... ,i_p}}$ are direct limits of
contractible complexes ($U_{i_0, ... ,i_p}$ is not necessarily
affine). So, by \cite[pp. 410, Lemma B.6]{TT}, $f_*\F|_{U_{i_0}}$,
..., $f_*\F|_{U_{i_n}}$, ..., $f_*\F|_{U_{i_0, ... ,i_n}}$ are
direct limits of contractible complexes. This implies that
$\mathfrak{C}^0(\mathfrak{U},\F)$, ...,
$\mathfrak{C}^n(\mathfrak{U},\F)$ are  direct limits of contractible
complexes.
\end{proof}

\begin{remark}
It is necessary to emphasis that, the direct image functor $f_*$ is
not exact in general case. But in some situations it preserves pure
exact sequences (Lemma \ref{pur111}). In the proof of Lemma
\ref{Cech0011} we don't need the exactness of $f_*$. Actually,
\cite[pp. 410, Lemma B.6]{TT} is the only property of $f_*$ that we
need in the proof of Lemma \ref{Cech0011}.
\end{remark}

Let $\CPX$ be the class of all pure acyclic complexes of
$\CO_X$-modules. A complex $\C$ of $\CO_X$-modules  is called
\textit{dg}-\textit{pure} \textit{injective} if  it belongs to
$\CPX^\perp$.
\begin{remark} \label{cotor1}
By using \cite[Theorem 4.10]{EEO16} and a similar argument used in
\cite[Proposition 3.4]{EG}), one can deduce that a complex
$\mathbf{G}=(\CG^i, \delta^i)$ of $\CO_X$-modules is dg-pure
injective if and only if

\begin{itemize}
\item [$(i)$]  For any $i\in \Z$, $\CG^i$ is a pure injective $\CO_X$-module.
\item [$(ii)$] For any pure acyclic  complex $\X$,
$\mathbf{Hom}^\bullet_{\CO_X}(\mathbf{X},\mathbf{G})$ is an acyclic
complex.
\end{itemize}
Recall that, for each pair $\X,\Y\in\C(X)$, the complex
$\mathbf{Hom}^\bullet_{\CO_X}(\mathbf{X},\Y)$ is  given by
$$(\mathbf{Hom}^\bullet_{\CO_X}(\mathbf{X},\Y)^n=\prod_{i\in\Z}{\text{Hom}}_{\CO_X}(\X^{i},\Y^{i+n}),\delta^n)$$
where
$$\delta^n ((f^i)_{i\in\Z})=\delta_\Y^{i+n} f^i-(-1)^n f^{i+1}(\delta_\X^{i}).$$
\end{remark}

Assume that
$\C(\mathrm{dg}\textmd{-}\mathrm{Pinj}X)\subseteq\C(\mathrm{Pinj}X)$
be the class of all dg-pure injective complexes. Then the following
result holds in the exact category $\C(X)$.

\begin{proposition}\label{main412}
The pair $(\CPX,\C(\mathrm{dg}\textmd{-}\mathrm{Pinj}X))$ is a
complete cotorsion theory.
\end{proposition}
\begin{proof}
\cite[Corollary 3.10.]{EGO14}
\end{proof}

Proposition \ref{main412} helps us to show that any pure acyclic
complex of pure injective $\CO_X$-modules is contractible. As a
consequence, we can prove that any complex of pure injective
$\CO_X$-modules is dg-pure injective.

\begin{lemma}\label{esikhan1}
Let $\X\in\CPX$ and $\G\in\CPX^\perp$. Then, for any $i>0$,
$$\Ext^i_{\C(X)}(\X,\G)=0.$$
\end{lemma}
\begin{proof}
By Proposition \ref{main412}, we can repeat the argument used in the
proof of \cite[Lemma 3.7]{HS13}.
\end{proof}
\begin{theorem}\label{h3}
Let $\C$ be a  complex of pure injective $\CO_X$-modules and $\X$ be
a direct limit of contractible complexes. Then
$\mathbf{Hom}^\bullet_{\CO_X}(\X,\C)$ is an acyclic complex.

\end{theorem}
\begin{proof}
Let $\X=\limt\X_j$ with $\X_j$ being contractible complex for each
$j\in J$. Then, we have the following degree-wise categorical pure
exact sequence
\[\xymatrix@C-0.7pc@R-0.9pc{0\ar[r]&\K\ar[r]&\bigoplus_{i\in I}\X_i\ar[r]&\X\ar[r]&0}.\]
Since $\C$ is a complex of pure injectives then we have the
following exact sequences of complexes
\[\xymatrix@C-0.7pc@R-0.9pc{0\ar[r]&\mathbf{Hom}^\bullet_{\CO_X}(\X,\C)\ar[r]&\mathbf{Hom}^\bullet_{\CO_X}(\bigoplus_{i\in I}\X_i,\C)\ar[r]&
\mathbf{Hom}^\bullet_{\CO_X}(\K,\C)\ar[r]&0}.\]

But, $\mathbf{Hom}^\bullet_{\CO_X}(\bigoplus_{i\in
I}\X_i,\C)\cong\prod_{i\in I}\mathbf{Hom}^\bullet_{\CO_X}(\X_i,\C)$
is the product of acyclic complexes of abelian groups and  so, it is
acyclic. By a similar method used in \cite[Proposition 5.3]{Sto14},
we can find a continuous chain $(\K_j)_{j\in J}$ of contractible
subcomplexes  of $\K$ such that $\K=\bigcup_{j\in J}\K_j$ and for
any $j$, $\Ext^1_{\C(X)}(\K_j,\C)=0$. Then, by \cite[Lemma 1]{ET01}
(the proof holds in  the exact category $\C(X)$),
$\Ext^1_{\C(X)}(\K,\C)=0$. Therefore,
$\mathbf{Hom}^\bullet_{\CO_X}(\K,\C)$ is an acyclic complex. This
shows that $\mathbf{Hom}^\bullet_{\CO_X}(\X,\C)$ is also acyclic.
So, by Remark \ref{cotor1}, $\F\in{}^\perp\C$.
\end{proof}
\begin{theorem}\label{esi}
Let $\C$ be a  complex of pure injective $\CO_X$-modules. Then
$\CPX\subseteq{}^\perp\C$.
\end{theorem}
\begin{proof}
Let $\F=(\CF^i,\pa^i_\F)$ be a pure acyclic complex of
$\CO_X$-modules and

\begin{align}\label{ss1}
  \xymatrix@C-0.7pc@R-0.9pc{0\ar[r]&\F\ar[r]&\mathfrak{C}^0(\mathfrak{U},\F)\ar[r]
&\mathfrak{C}^1(\mathfrak{U},
\F)\ar[r]&\cdots\ar[r]&\mathfrak{C}^{n-1}(\mathfrak{U},\F)\ar[r]&\mathfrak{C}^{n}(\mathfrak{U},\F)\ar[r]&
0}
\end{align}
be its $\check{\mathrm{C}}$ech resolution.  By the proof of
\cite[Lemma III. 4.2]{H}, we deduce that for any $i\in\Z$,  the
complex

\begin{align}\label{ss1110}
  \xymatrix@C-0.9pc@R-0.9pc{0\ar[r]&\CF^i\ar[r]&\mathfrak{C}^0(\mathfrak{U},\CF^i)\ar[r]
&\mathfrak{C}^1(\mathfrak{U},
\CF^i)\ar[r]&\cdots\ar[r]&\mathfrak{C}^{n-1}(\mathfrak{U},\CF^i)\ar[r]&\mathfrak{C}^{n}(\mathfrak{U},\CF^i)\ar[r]&
0}
\end{align}
is pure acyclic. This implies that \eqref{ss1} is a degree-wise pure
acyclic complex of complexes. Moreover, by Lemma \ref{Cech0011}, for
any $0\leq p\leq n$, $\mathfrak{C}^p(\mathfrak{U},\F)$ is a direct
limit of contractible complexes and  hence by Lemma \ref{h3},
$\mathbf{Hom}^\bullet_{\CO_X}(\mathfrak{C}^p(\mathfrak{U},\F),\C)$
is an acyclic complex.  We can break \eqref{ss1} to the following
degree-wise pure exact sequences ($0\leq p\leq n$),
\[\xymatrix@C-0.7pc@R-0.9pc{0\ar[r]&\K^{p-1}\ar[r]&\mathfrak{C}^p(\mathfrak{U},\F)\ar[r]&\mathbf{P}^p\ar[r]&0.}\]
By Lemma \ref{h3},
$\mathbf{Hom}^\bullet_{\CO_X}(\mathfrak{C}^p(\mathfrak{U},\F),\C)$,
$\mathbf{Hom}^\bullet_{\CO_X}(\mathbf{P}^p,\C)$ are  acyclic
complexes. Therefore, for any $p\geq 1$,
$\mathbf{Hom}^\bullet_{\CO_X}(\mathbf{K}^{p-1},\C)$ is an acyclic
complex. Especially,
$\mathbf{Hom}^\bullet_{\CO_X}(\mathbf{K}^0,\C)=\mathbf{Hom}^\bullet_{\CO_X}(\F,\C)$
is  acyclic. Therefore, by Remark \ref{cotor1}, $\F\in{}^\perp\C$.
\end{proof}
\begin{corollary}\label{abtin1}\label{yas1}
Any pure acyclic complex of pure injective $\CO_X$-modules  is
contractible.
\end{corollary}
\begin{corollary}\label{abtin2}
Any complex of pure injective $\CO_X$-modules is dg-pure injective.
\end{corollary}
\section{The pure derived category of $\CO_X$-modules}
Let $\CT$ be a triangulated category. A triangulated subcategory
$\CS$ of $\CT$ is called \textit{thick} if it is closed under direct
summands. Assume that $\CS$ is a thick  subcategory of $\CT$. The
right orthogonal of $\CS$ in $\CT$ is defined  by
\[ \CS^\perp=\{ X \in \CT \ | \ \Hom_\CT(S,X)=0, \ {\rm{for \ all}} \ S \in \CS \}.\]
By  \cite[Theorem 9.1.13]{N01}, the inclusion $\CS \lrt \CT$ has a
right adjoint if and only if for any object $\mathbf{T}$ in $\CT$,
there is a distinguished triangle $$ \xymatrix@C-0.5pc{\mathbf{S}
\ar[r] & \mathbf{T} \ar[r] & \mathbf{S}' \ar[r] &\Sigma \mathbf{S} }
$$ in $\CT$ where $\mathbf{S} \in \CS$ and $\mathbf{S}' \in
\CS^\perp$. This triangle is unique up to isomorphism and the right
adjoint of $\CS \lrt \CT$ maps $\mathbf{T}$ to $\mathbf{S}$. This
situation implies an equivalence $\CS^\perp \lrt \CT/\CS$  of
triangulated categories. Let $\K(X)$ be the homotopy category of
$\mathfrak{Qco}X$, $\KPACX$ be its thick subcategory consisting of
pure acyclic complexes and $\KPIX$ be the essential image of the
homotopy category of pure injective $\CO_X$-modules in $\K(X)$
($\KPIX$ is closed under isomorphism). The following Theorem, is the
main result of this paper.

\begin{theorem}\label{esi12}
$\KPIX=\KPACX^\perp$.
\end{theorem}

\begin{proof}

Assume that $\C$ is an object of $\KPIX$. Then, by Corollary
\ref{abtin2} it is a dg-pure injective complex and so by Remark
\ref{cotor1},  $\C\in\KPACX^\perp$. Conversely, let
$\X\in\KPACX^\perp$. By Proposition \ref{main412}, there is a
degree-wise pure exact sequence
$\xymatrix@C-0.7pc@R-0.9pc{0\ar[r]&\mathbf{X}\ar[r]^f&\mathbf{I}\ar[r]&\mathbf{P}\ar[r]&0}$
of complexes, where $\mathbf{P}\in\CPX$ and
$\mathbf{I}\in\CPX^\perp\subseteq\C(\mathrm{Pinj}X)$. Hence,  we
obtain a triangle
\begin{align}\label{331}
\xymatrix@C-0.7pc@R-0.9pc{\X\ar[r]&\mathbf{I}\ar[r]&\mathbf{P}_0\ar[r]&\Sigma\X}
\end{align}
where  $\mathbf{P}_0=\mathrm{Cone}(f)\in\KPACX$ and
$\mathbf{I}\in\mathbf{K}(\mathrm{Pinj}X)$. By applying the
cohomological $\Hom_{\K(X)}(\textmd{-},\X)$ the proof completes.
\end{proof}
\begin{corollary}
There is an equivalence $\KPIX\lrt\DPACX$.
\end{corollary}
\begin{proof}
By the proof of Theorem \ref{esi12},   we have  a distinguished
triangle
\begin{align}\label{331}
\xymatrix@C-0.7pc@R-0.9pc{\X\ar[r]&\mathbf{I}\ar[r]&\mathbf{P}_0\ar[r]&\Sigma\X}
\end{align}
where  $\mathbf{P}_0\in\KPACX$ and $\mathbf{I}\in
\KPACX^\perp\subseteq\mathbf{K}(\mathrm{Pinj}X)$.  Therefore,
$\KPACX\lrt\K(X)$ admits a right adjoint and hence, we have  an
equivalence $\KPIX=\KPACX^\perp\lrt\DPACX$ of triangulated
categories.
\end{proof}
Recently, the existence of adjoint pairs of functors between
homotopy categories have been interested  (see, \cite{N14},
\cite{K12}, \cite{Sto14} \cite{K14}, \cite{Gi15}). The following
result is in this direction.
\begin{corollary}
The inclusion $\KPIX\lrt\K(X)$ admits a left adjoint.
\end{corollary}

\vspace{.5cm}

\textbf{Acknowledgements:} The author is deeply grateful to the
referee for his/her careful reading of the manuscript. This work was
supported by Iran National Science Foundation: INSF.

\end{document}